\documentclass[a4paper, 12pt, reqno]{amsart}
\usepackage[utf8]{inputenc}
\usepackage[english]{babel}
\usepackage{amsfonts}
\usepackage{amsthm}
\usepackage{amsmath}
\usepackage{amscd}
\usepackage{mathrsfs}
\usepackage{amssymb}
\usepackage{euscript}
\usepackage{color}
\usepackage{tikz}
\usepackage{graphics}
\usepackage{epsfig}
\usepackage{graphicx}
\usepackage{subcaption}
\usepackage[all]{xy}
\usepackage[colorlinks=true, linkcolor=blue, citecolor=blue]{hyperref}

\usetikzlibrary{shapes, decorations, calc, fit, positioning}
\tikzset{
    auto,
    point/.style = {circle, draw, inner sep=0.02cm, fill, node contents={}},
}

\advance\hoffset-20mm \advance\textwidth40mm

\DeclareMathOperator{\Span}{\mathrm{span}}

\DeclareMathOperator{\Cone}{\mathrm{Cone}}
\DeclareMathOperator{\conv}{\mathrm{Conv}}
\DeclareMathOperator{\diag}{\mathrm{diag}}
\DeclareMathOperator{\Hom}{\mathrm{Hom}}
\DeclareMathOperator{\Aut}{\mathrm{Aut}}
\DeclareMathOperator{\A}{\mathbb{A}}
\DeclareMathOperator{\CC}{\mathbb{C}}
\DeclareMathOperator{\GG}{\mathbb{G}}
\DeclareMathOperator{\PP}{\mathbb{P}}
\DeclareMathOperator{\QQ}{\mathbb{Q}}

\DeclareMathOperator{\ZZ}{\mathbb{Z}}
\DeclareMathOperator{\Sgmone}{\Sigma (1)}
\DeclareMathOperator{\lSc}{\langle}
\DeclareMathOperator{\rSc}{\rangle}
\def\addgr{\GG_a^3}
\def\OO{\mathscr{O}}
\def\unipgr{\GG_a}

\theoremstyle{definition}

\newtheorem{example}{Example}

\theoremstyle{plain}

\newtheorem{lemma}{Lemma}

\newtheorem{proposition}{Proposition}
\newtheorem{theorem}{Theorem}

\makeatletter
\@namedef{subjclassname@2020}{\textup{2020} Mathematics Subject Classification}
\makeatother
\begin{document}
\sloppy

\title[NON-PROJECTIVE COMPLETIONS OF AFFINE SPACES]{SMOOTH NON-PROJECTIVE EQUIVARIANT COMPLETIONS OF AFFINE SPACES}

\author[KIRILL SHAKHMATOV]{KIRILL SHAKHMATOV}
\date{\today}
\thanks{The author was supported by the grant RSF-19-11-00172}

\address{Moscow State University, Faculty of Mechanics and Mathematics, Department of Higher Algebra,
Leninskie Gory 1, Moscow, 119991 Russia}
\address{National Research University Higher School of Economics, Faculty of Computer Science,
Pokrovsky Boulevard 11, Moscow, 109028 Russia}
\email{bagoga@list.ru}

\begin{abstract}
In this paper we construct an equivariant embedding of the affine space \( \A^n \) with the translation group action into a complete non-projective algebraic variety \( X \) for all \( n \geq 3 \). The theory of toric varieties is used as the main tool for this construction. In the case of \( n = 3 \) we describe the orbit structure on the variety \( X \).
\end{abstract}

\maketitle

\section*{Introduction}

Let \( X \) be a complete irreducible algebraic variety of dimension \( n \) over the field of complex numbers \( \CC \). Consider a group \( \GG_a^n = (\CC^n, +) \) that acts on the affine space \( \A^n \) by translations. An open embedding \( \phi : \A^n \to X \) with an extension of the action \( \GG_a^n \times \A^n \to \A^n \) to a regular action \( \GG_a^n \times X \to X \) is called an equivariant embedding of affine space. Clearly, to determine an equivariant embedding means to define an effective regular action of group \( \GG_a^n \) on \( X \) with an open orbit. Such actions are called additive.

Study of equivariant completions of the affine space \( \A^n \) was started by Hassett and Tschinkel \cite{HT-GoEC}. It was shown that additive actions on the projective space \( \PP^n \) correspond to finite-dimensional local commutative associative algebras. Further results on additive actions can be found in \cite{A-FVaEC, AR-AAoTV} and in bibliography therein.

\smallskip

The category of equivariant completions of affine spaces is the object of interest. It is well-known that any smooth complete variety of dimension one or two is projective \cite[Chapter II.4]{H-AG}. In this paper we construct a non-projective completion of affine space, starting from dimension three. Our main tool is the theory of toric varieties. As a variety \( X \) we use one of the standard examples of a complete non-projective toric variety. Its fan is formed by cones over polytopes of a subdivision of a simplex \cite[Example 2.1.4]{BP-TT}. In Section~\ref{main results} we prove that \( X \) admits an additive action and describe its orbits. We acquire this action by using Cox rings \cite{C-HCR} and the quotient construction for toric varieties. When the action is defined on the spectrum of the Cox ring by explicit formulas, we show that this action induces an additive action on \( X \).

\smallskip

The author is grateful to his supervisor Ivan Arzhantsev for posing the problem, constant support and references, and Yulia Zaitseva for useful discussions and commentaries.


\section{Preliminaries}

In this section we recall main definitions and constructions and formulate propositions that will be used later. Let \( N \simeq \ZZ^n \) be a lattice of rank \( n \) and \( M = \Hom(N, \ZZ) \) be its dual lattice. We denote by \( N_{\QQ} \) the space \( N \otimes_{\ZZ} \QQ \simeq \QQ^n \) and by \( M_{\QQ} \) the space \( M \otimes_{\ZZ} \QQ \). Let \( \lSc \cdot, \cdot \rSc : M_{\QQ} \times N_{\QQ} \to \QQ \) be the canonical pairing. We define a \emph{polyhedral cone in} \( N \) as a subset
\[ \Cone(v_1, \dots, v_r) = \{ \lambda_1 v_1 + \dots + \lambda_r v_r \ | \ \lambda_1, \dots, \lambda_r \in \QQ_{\geq 0} \} \]
for some \( v_1, \dots, v_r \in N \). A cone is called \emph{strongly convex} if it does not contain subspaces of positive dimension. A finite set \( \Sigma \) of cones in \( N \) is called a \emph{fan} if any face of any cone in \( \Sigma \) is also a cone in \( \Sigma \) and the intersection of two cones in \( \Sigma \) is a face of each.

Each fan \( \Sigma \) in \( N \) corresponds to a toric variety \( X_{\Sigma} \), i.e.~a normal algebraic variety equipped with an effective action of the torus \( (\CC^{\times})^n \) with an open dense orbit. Let us recall the construction of variety \( X_{\Sigma} \). Let \( \sigma \) be a cone in \( N \). The cone
\[ \sigma^{\vee} = \{ u \in M_{\QQ} \ | \lSc u,v \rSc \geq 0 \ \forall v \in \sigma \} \]
in \( M \) is called \emph{dual to} \( \sigma \). An affine variety \( U_{\sigma} \) is defined as the spectrum of the finitely generated algebra
\[ \CC[ \sigma^{\vee} \cap M ] = \bigoplus_{m \in \sigma^{\vee} \cap M} \CC \chi^m, \]
which has the multiplication given by \( \chi^{m} \cdot \chi^{m'} = \chi^{m+m'} \). A variety \( X_{\Sigma} \) is covered by affine charts \( U_{\sigma} \) for all \( \sigma \) in \( \Sigma \), so that two charts \( U_{\tau} \) and \( U_{\tau'} \) are glued along their principal open subset \( U_{\tau \cap \tau'} \). An action of torus \( T = (\CC^{\times})^n \) is defined as follows. If a point \( t \in T \) corresponds to a group homomorphism \( M \to \CC^{\times} \) and a point \( x \in U_{\sigma} \) corresponds to a semigroup homomorphism \( \sigma^{\vee} \cap M \to \CC \), then the point \( t \cdot x \) corresponds to a homomorphism \( \sigma^{\vee} \cap M \to \CC \) such that \( u \mapsto t(u)x(u) \). A detailed description of this construction and the main properties of toric varieties can be found, for example, in \cite{F-ItTV, CLS-TV}. Let us recall some of these properties.

A toric variety \( X_{\Sigma} \) is smooth if and only if every cone \( \sigma \in \Sigma \) is generated by a subset of some basis of the lattice \( N \) \cite[Chapter~2.1]{F-ItTV}. It is known that orbits of toric variety \( X_{\Sigma} \) are in bijection with cones of the fan \( \Sigma \). A cone \( \sigma \in \Sigma \) of dimension \( k \) corresponds to an orbit \( O(\sigma) = \Hom(\sigma^{\perp} \cap M, \CC^{\times}) \simeq (\CC^{\times})^{n-k} \), where
\[ \sigma^{\perp} = \{ u \in M_{\QQ} \ | \ \lSc u,v \rSc = 0 \ \forall v \in \sigma \} .\]
Every group homomorphism \( \phi: \sigma^{\perp} \cap M \to \CC^{\times} \) can be extended to a semigroup homomorphism \( \sigma^{\vee} \cap M \to \CC \) if we define \( \phi(u) = 0 \) for any \( u \in \sigma^{\vee} \setminus \sigma^{\perp} \). Thus, an orbit \( O(\sigma) \) embeds into the chart \( U_{\sigma} \). We denote by \( V(\sigma) \) the closure of \( O(\sigma) \). For cones \( \sigma \) and \( \tau \) we write \( \tau \preceq \sigma \) if \( \tau \) is a face of \( \sigma \).

The next proposition is proved in \cite[Theorem 3.2.6]{CLS-TV} and \cite[Chapter 3.1]{F-ItTV}.

\begin{proposition}
For a toric variety \( X_{\Sigma} \) next statements hold:
\begin{enumerate}
\item for a cone \( \sigma \in \Sigma \) the affine chart \( U_{\sigma} \) is the union of orbits \( O(\tau) \) for \( \tau \preceq \sigma \):
\[ U_{\sigma} = \bigsqcup_{\tau \preceq \sigma} O(\tau); \]
\item \( \tau \preceq \sigma \) if and only if \( O(\sigma) \subseteq V(\tau) \), and
\[ V(\tau) = \bigsqcup_{\tau \preceq \sigma} O(\sigma); \]
\item if \( N_{\tau} \) is a subgroup of \( N \) generated by \( \tau \cap N \) and \( \psi : N \to N / N_{\tau} \) is the canonical projection, then the toric variety \( V(\tau) \) corresponds to the fan \( \{ \psi(\sigma) \ | \ \tau \preceq \sigma \} \) in the lattice \( N / N_{\tau} \).
\end{enumerate}
\end{proposition}

\begin{example}
Let us fix a basis \( e_1, \dots, e_n \) in the lattice \( N \) and consider a fan \( \Sigma \) consisting of a cone \( \sigma = \Cone(\{ e_i \}_{i=1}^n) \) and all its faces. Let \( e_1^*, \dots, e_n^* \) be the dual basis in the lattice \( M \). Since \( \sigma^{\vee} = \Cone(\{ e_i^* \}_{i=1}^n) \), we have \( X_{\Sigma} = U_{\sigma} = \CC^n \). Consider a subset \( I \subseteq \{ 1, \dots, n \} \) and a cone \( \sigma_I = \Cone(\{ e_i \}_{i \in I}) \in \Sigma \). It is easy to show that
\[ O(\sigma_I) = \Hom( \Span(\{ e_i^* \}_{i \notin I}) \cap M, \CC^{\times} ) = \{ (x_1, \dots, x_n) \ | \ x_i = 0 \ \forall i \in I; x_i \neq 0 \ \forall i \notin I \}. \]
\end{example}

Let us study morphisms that are compatible with the structure of toric varieties. Consider a fan \( \Sigma' \) in a lattice \( N' \). Let \( f : N' \to N \) be a lattice homomorphism. Denote by \( f_{\QQ} \) its extension to a linear map \( N'_{\QQ} \to N_{\QQ} \), and let \( f^*: M \to M' = \Hom(N', \ZZ) \) be the homomorphism dual to \( f \). The homomorphism \( f \) is called \emph{compatible with the fans} \( \Sigma' \) and \( \Sigma \) if for every cone \( \sigma' \in \Sigma' \), there exists a cone \( \sigma \in \Sigma \) such that \( f_{\QQ}(\sigma') \subseteq \sigma \). In this case for every cone \( \sigma' \in \Sigma' \), there exists a cone \( \sigma \in \Sigma \) such that the restriction of \( f^* \) to \( \sigma^{\vee} \cap M \) defines a semigroup homomorphism \( \sigma^{\vee} \cap M \to (\sigma')^{\vee} \cap M' \). Thus we obtain a homomorphism of algebras \( \CC[ \sigma^{\vee} \cap M ] \to \CC[ (\sigma')^{\vee} \cap M' ] \) and the dual morphism of varieties \( \phi_f^{\sigma'} : U_{\sigma'} \to U_{\sigma} \). They glue to a morphism \( \phi_f : X_{\Sigma'} \to X_{\Sigma} \) which is called \emph{toric}.

\begin{proposition} \label{prop toric morphisms}
\cite[Lemma 3.3.21]{CLS-TV} Let \( {\phi = \phi_f : X_{\Sigma'} \to X_{\Sigma}} \) be a toric morphism. For a cone \( \sigma' \in \Sigma' \) let \( \sigma \) be the minimal cone of the fan \( \Sigma \) that contains \( f_{\QQ}(\sigma') \). Then \( \phi( O(\sigma') ) \subseteq O(\sigma) \), \( \phi( V(\sigma') ) \subseteq V(\sigma) \), and the induced morphism \( \phi \mid_{V(\sigma')} : V(\sigma') \to V(\sigma) \) is toric.
\end{proposition}

There exists a class of fans that can be described in terms of convex polytopes in \( M_{\QQ} \). A fan \( \Sigma \) is called \emph{complete} if \( \bigcup_{\sigma \in \Sigma} \sigma = N_{\QQ} \). Recall that a toric variety \( X_{\Sigma} \) is complete if and only if its fan \( \Sigma \) is complete \cite[Chapter 2.4]{F-ItTV}. Let \( P \) be a convex polytope in \( M_{\QQ} \) such that the interior of \( P \) contains \( 0 \). Assume that every vertex of \( P \) belongs to \( M \). For every face \( Q \) of \( P \) we define a cone
\[ \sigma_Q = \{ v \in N_{\QQ} \ | \ \lSc u, v \rSc \leq \lSc u', v \rSc \forall u \in Q \ \forall u' \in P \}. \]
Collection of cones \( \sigma_Q \), as \( Q \) varies over all faces of \( P \), forms a fan \( \Sigma_P \) \cite[Chapter 1.5]{F-ItTV}. This fan is called \emph{dual to polytope} \( P \). It is easily seen that the fan \( \Sigma_P \) is complete. It is well-known that a complete toric variety \( X_{\Sigma} \) is projective if and only if the fan \( \Sigma \) is dual to some convex polytope \cite[Proposition 7.2.9]{CLS-TV}. These results make possible to construct examples of complete non-projective toric varieties starting from dimension 3, see \cite[Example 2.1.4]{BP-TT} and \cite[Chapter 1.5]{F-ItTV}.

\medskip

Now we recall the construction of Cox ring for toric varieties, see \cite{C-HCR} and \cite[Chapter 5]{CLS-TV}. We need some facts about quotients in algebraic geometry. A morphism of algebraic varieties \( \pi : Y \to X \) is called \emph{affine} if for any affine open subset \( U \subseteq X \) its preimage \( \pi^{-1} (U) \) is affine. Let \( G \) be an affine algebraic group that acts on a variety \( Y \), and let \( \pi : Y \to X \) be a morphism constant on \( G \)-orbits. The morphism \( \pi \) is called a \emph{good categorical quotient} if \( \pi \) is affine and for any open \( U \subseteq X \) the dual homomorphism of function algebras \( \mathscr{O}_X (U) \to \mathscr{O}_Y ( \pi^{-1} (U) ) \) induces an isomorphism \( \mathscr{O}_X (U) \simeq \mathscr{O}_Y ( \pi^{-1} (U) )^G \). In this case \( X \) is denoted by \( Y/\!/G \).

Let us recall the main properties of good categorical quotients. It is known that \( \pi \) maps disjoint closed \( G \)-invariant subsets of variety \( Y \) to disjoint closed subsets of \( Y/\!/G \). Thus, the preimage \( \pi^{-1}(p) \) of any point \( p \in Y/\!/G \) contains a unique closed \( G \)-orbit, so there is a one-to-one correspondence between closed \( G \)-orbits in \( Y \) and points of variety \( Y/\!/G \) \cite[Proposition 5.0.7]{CLS-TV}.

\begin{proposition} \label{prop geometric factor}
\cite[Proposition 5.0.8]{CLS-TV} For a good categorical quotient \( \pi : Y \to Y/\!/G \) the following are equivalent:
\begin{enumerate}
\item any \( G \)-orbit in \( Y \) is closed;
\item if \( x,y \in Y \), then \( \pi (x) = \pi (y) \) if and only if \( x \) and \( y \) are in the same \( G \)-orbit;
\item the morphism \( \pi \) induces a bijection \( \{ G \)-orbits in \( Y \} \leftrightarrow \{  \)points of \( Y/\!/G \} \).
\end{enumerate}
\end{proposition}

A good categorical quotient \( \pi : Y \to Y/\!/G \) is called \emph{geometric} if it satisfies the conditions of Proposition \ref{prop geometric factor}. In this case we write \( Y/\!/G = Y/G \).

\smallskip

Consider a fan \( \Sigma \) in the lattice \( N \). We denote by \( \sigma'(1) \) (or \( \Sigma'(1) \)) the set of rays of a cone \( \sigma' \) (or a fan \( \Sigma' \)). For any ray \( \rho \in \Sgmone \) the generating vector of the semigroup \( \rho \cap N \) is denoted by~\( p_{\rho} \). Assume that \( \Span_{\QQ} \{ p_\rho \}_{ \rho \in \Sgmone } = N_{\QQ} \) and let \( \{ f_i \}_{i=1}^n \) be a basis of the dual lattice \( M \). Consider a group
\[ G_{\Sigma} = \{ (t_\rho) \in (\CC^{\times})^{\Sgmone} \ | \displaystyle\prod_{\rho \in \Sgmone} t_\rho^{\lSc f_i, p_\rho \rSc} = 1, 1 \leq i \leq n \} \subseteq (\CC^{\times})^{\Sgmone}, \label{group formula} \tag{\( \ast \)} \]
that acts on \( \CC^{\Sgmone} \) by coordinate-wise multiplication. A subset \( C \subseteq \Sgmone \) is called a \emph{primitive collection} if for any cone \( \sigma \in \Sigma \) the subset \( C \) is not contained in \( \sigma(1) \), and for any proper subset \( C' \subseteq C \), there is a cone \( \sigma \in \Sigma \) such that \( C' \subseteq \sigma (1) \). The set of all primitive collections is denoted by \( \mathscr{C} (\Sigma) \).

Let \( Z = Z(\Sigma) = \underset{C \in \mathscr{C} (\Sigma)}{\bigcup} \mathbb{V}(x_\rho \ | \ \rho \in C) \subseteq \CC^{\Sgmone} \), where \( \mathbb{V}(x_\rho \ | \ \rho \in C) \) is a set of points \( (x_\rho) \in \CC^{\Sgmone} \) such that \( x_\rho = 0 \) for all \( \rho \in C \). The set \( Z \) is invariant under the action of the group \( G_{\Sigma} \). Thus, there is an induced action of \( G_{\Sigma} \) on \( \CC^{\Sgmone} \setminus Z \). Let \( \{ e_\rho \ | \ \rho \in \Sgmone \} \) be the standard basis of the lattice \( \ZZ^{\Sgmone} \). For any \( \sigma \in \Sigma \) we define \( \tilde{\sigma} =  \Cone(e_\rho \ | \ \rho \in \sigma (1)) \). Cones \( \tilde{\sigma} \) and all their faces form a fan \( \tilde{\Sigma} = \{ \tau \ | \ \exists \sigma \in \Sigma : \tau \preceq \tilde{\sigma} \} \).

\smallskip

The next proposition is proved in \cite[Proposition 5.1.9, Theorem 5.1.11]{CLS-TV}.

\begin{proposition}
\
\begin{enumerate}
\item The variety \( \CC^{\Sgmone} \setminus Z \) is a toric variety that corresponds to the fan \( \tilde{\Sigma} \).
\item The map \( e_\rho \mapsto p_\rho \) induces a map of lattices \( \ZZ^{\Sgmone} \to N \), which is compatible with the fans \( \tilde{\Sigma} \) in \( \ZZ^{\Sgmone} \) and \( \Sigma \) in \( N \).
\item The corresponding toric morphism \( \pi : \CC^{\Sgmone} \setminus Z \to X_{\Sigma} \) is a good categorical quotient by the action of the group \( G_{\Sigma} \), so that \( (\CC^{\Sgmone} \setminus Z) /\!/ G_{\Sigma} \simeq X_{\Sigma} \). Moreover, \( \pi \) is a geometric quotient if and only if the fan \( \Sigma \) is simplicial.
\end{enumerate}
\end{proposition}

As we can see, \( \tilde{\Sigma} \) is a subset of fan \( {\Sigma}_0 \) formed by a cone \( \Cone(e_\rho \ | \ \rho \in \Sgmone) \) and all its faces. The fan \( \Sigma_0 \) corresponds to the toric variety \( \CC^{\Sgmone} \). The identical map \( \ZZ^{\Sgmone} \to \ZZ^{\Sgmone} \) is compatible with the fans \( \tilde{\Sigma} \) and \( {\Sigma}_0 \) correspondingly, and its toric morphism is an embedding \( \CC^{\Sgmone} \setminus Z \hookrightarrow \CC^{\Sgmone} \). By Proposition \ref{prop toric morphisms} for any cone \( \sigma \in \tilde{\Sigma} \) we have \( O(\sigma)_{\tilde{\Sigma}} = O(\sigma)_{\Sigma_0} \setminus Z \).

\smallskip

Assume now that a toric variety \( X \) with an acting torus \( T \) is equipped with an additive action \( \GG_a^n \times X \to X \). We say that this action is \emph{normalized by torus} if the subgroup \( {\GG_a^n \subseteq \Aut(X)} \) is normalized by the subgroup \( T \subseteq \Aut(X) \). In this case \( T \) permutes \( \GG_a^n \)~-~orbits: for all \( t \in T \) and \( x \in X \) we have \( t \GG_a^n x = t \GG_a^n t^{-1} t x = \GG_a^n t x \). Therefore, an open dense \( \GG_a^n \)-orbit \( \OO \) is \( T \)-invariant. In addition, irreducible components of the closed subvariety \( X \setminus \OO \) are toric subvarieties of codimension 1.


\section{Main Results} \label{main results}

Let us fix a basis \( e_1, \dots, e_n \) in a lattice \( N \) of rank \( n \). Consider vectors
\[ {a_1 = e_1}, \dots, \ {a_n = e_n}, \ {a_0 = -e_1 - \cdots - e_n}, \ {b_1 = a_0 + a_1}, \ b_2 = a_0 + a_2, \dots, \ b_n = a_0 + a_n. \]
We denote by \( \conv(A) \) the convex hull of a set \( A \subseteq N_{\QQ} \). Let us consider a simplex \( P = \conv(a_0, a_1, \dots, a_n) \). Its faces of codimension \( 1 \) are \( {A_i = \conv(a_0, a_1, \dots, \widehat{a_i}, \dots, a_n),} \ {i = 0, \dots, n} \). Let \( b'_i = \frac{1}{2}b_i \). As we can see, the point \( b'_i \) lies in the segment \( [a_0, a_i] \), which is contained in the faces \( A_j, \ {j = 1, \dots, \widehat{i}, \dots, n} \) and not contained in the faces \( A_0 \) and \( A_i \). We subdivide each face \( A_i, \ i = 1, \dots, n \) into \( n \) simplices \( A'_{i,j} \) of dimension \( n-1 \), which are defined as follows (the illustration of the case \( n = 3 \) is represented on Figure~\ref{pic subdivision}):

\[ A'_{1,1} = \conv(a_2, \dots, a_n, b'_2), \ A'_{1,2} = \conv(a_3, \dots, a_n, b'_2, b'_3), \dots, \]
\[ A'_{1,n-1} = \conv(a_n, b'_2, \dots, b'_n), \ A'_{1,n} = \conv(b'_2, \dots, b'_n, a_0); \]
\[ \dots\dots \]
\[ A'_{n-1,1} = \conv(a_n, a_1, \dots, a_{n-2}, b'_n), \ A'_{n-1,2} = \conv(a_1, \dots, a_{n-2}, b'_n, b'_1), \dots, \]
\[ A'_{n-1,n-1} = \conv(a_{n-2}, b'_n, b'_1, \dots, b'_{n-2}), \ A'_{n-1,n} = \conv(b'_n, b'_1 \dots, b'_{n-2}, a_0); \]

\[ A'_{n,1} = \conv(a_1, \dots, a_{n-1}, b'_1), \ A'_{n,2} = \conv(a_2, \dots, a_{n-1}, b'_1, b'_2), \dots, \]
\[ A'_{n,n-1} = \conv(a_{n-1}, b'_1, \dots, b'_{n-1}), \ A'_{n,n} = \conv(b'_1, \dots, b'_{n-1}, a_0). \]

\begin{figure}[h]
\begin{tikzpicture}
\node (1) [label = below: $a_1$, point, yshift = 0cm, xshift = 0cm];
\node (2) [label = left: $a_2$, point, left = of 1, yshift = 1cm, xshift = -1cm];
\node (3) [label = right: $a_3$, point, right = of 1, yshift = 1cm, xshift = 1cm];
\node (4) [label = above: $a_0$, point, above = of 1, yshift = 2.5cm, xshift = 0cm];
\node (5) [label = right: $b'_3$, point, right = of 1, yshift = 2.25cm, xshift = 0cm];
\node (6) [label = below left: $b'_1$, point, above = of 1, yshift = 0.6cm, xshift = 0cm];
\node (7) [label = left: $b'_2$, point, left = of 1, yshift = 2.25cm, xshift = 0cm];
\path (1) edge (2);
\path (1) edge (3);
\path (1) edge (4);
\path (2) edge (4);
\path (3) edge (4);
\path (7) edge (6);
\path (6) edge (5);
\path[dashed] (2) edge (3);
\path[dashed] (7) edge (5);
\path (1) edge (5);
\path (2) edge (6);
\path[dashed] (3) edge (7);
\end{tikzpicture}
\caption{}
\label{pic subdivision}
\end{figure}

Consider the fan \( \Sigma \) in \( N \) over this subdivision of the simplex \( P \).

\begin{theorem} \label{main theorem}
\
\begin{enumerate}
\item The variety \( X = X_{\Sigma} \) is smooth, complete and non-projective.
\item The variety \( X \) admits an additive action normalized by the acting torus.
\end{enumerate}
\end{theorem}

\begin{proof}
It is obvious that the fan \( \Sigma \) is complete. One can easily prove that any set of \( n \) vectors generating a cone of maximal dimension from \( \Sigma \) is a basis of the lattice \( N \). Thus, the variety \( X \) is smooth. The next lemma finishes the proof of the first assertion of theorem.

\begin{lemma}
The fan \( \Sigma \) is not dual to any polytope.
\end{lemma}

\begin{proof}
Assume that the fan \( \Sigma \) is dual to some polytope \( Q \). Since \( \dim (\sigma_R) + \dim (R) = n \) for any face \( R \) of the polytope \( Q \), and the subdivision of the simplex \( P \) contains \( n^2 + 1 \) simplices of dimension \( n - 1 \), the polytope \( Q \) has \( n^2 + 1 \) vertices \( u_1, \dots, u_{n^2 + 1} \). These vertices correspond to cones of maximal dimension
\[ \sigma_i = \sigma_{u_i} = \{ v \in N_{\QQ} \ | \ \lSc u_i, v \rSc \leq \lSc u', v \rSc \ \forall u' \in Q \}, \ \ i = 1, \dots, n^2 + 1. \]
Consider a function \( \phi(x) = \min_k u_k(x) \) on \( N_{\QQ} \). By definition of \( \sigma_i \) we have \( \phi \! \mid_{\sigma_i} = u_i \). Since the vectors \( a_1, b_n, a_n \) lie in the same cone \( \omega_j = \sigma_{n-1,1} \), we obtain
\[ \phi(a_1) + \phi(b_n) - \phi(a_n) = u_j(a_1) + u_j(b_n) - u_j(a_n)  = u_j(a_1 + b_n - a_n). \]
By an equation \( a_1 + b_n - a_n = b_1 \) we get \(  u_j(a_1 + b_n - a_n) = u_j(b_1) \). The vector \( b_1 \) is not contained in the cone \( \omega_j \), but lies in some other cone \( \omega_l \) of the fan \( \Sigma \). Therefore, \( \phi(b_1) = u_l(b_1) < u_j(b_1) \). Finally, we have \( \phi(a_1) + \phi(b_n) - \phi(a_n) > \phi(b_1) \). Similarly, since \( a_i + b_{i-1} - a_{i-1} = b_i \),  we obtain \( \phi(a_i) + \phi(b_{i-1}) - \phi(a_{i-1}) > \phi(b_i) \) for all \( i = 2, \dots, n \). By summing up these inequalities we get a contradiction \( \phi(b_1) + \cdots + \phi(b_n) > \phi(b_1) + \cdots + \phi(b_n) \).
\end{proof}

Now we proceed to a proof of the second assertion. Let \( b_0 = a_0 \). The set \( \Sgmone \) consists of \( 2n+1 \) rays \( \rho_i = \Cone(b_i), \ 0 \leq i \leq n \) and \( \rho'_j = \Cone(a_j), \ 1 \leq j \leq n \). We order coordinates \( (x_0, x_1, \dots, x_n, x'_1, \dots, x'_n) \) on \( \CC^{\Sgmone} \) in accordance with the order of rays \( (\rho_0, \rho_1, \dots, \rho_n, \rho'_1, \dots, \rho'_n) \). Let \( Y = \CC^{\Sgmone} \setminus Z \), where \( Z = Z(\Sigma) \). We write conditions~(\ref{group formula}) in terms of diagonal characters \( \omega_0, \omega_1, \dots, \omega_n, \omega'_1, \dots, \omega'_n \) of quasitorus \( G = G_{\Sigma} \):
\[
    \left \{
    \begin{array}{c}
         - \omega_0 - \omega_2 - \cdots - \omega_n + \omega'_1 = 0 \\
         - \omega_0 - \omega_1 - \omega_3 - \cdots - \omega_n + \omega'_2 = 0 \\
         \dots \\
         - \omega_0 - \omega_1 - \cdots - \omega_{n-1} + \omega'_n = 0 \\
    \end{array}
    \right . \implies
    \left \{
    \begin{array}{c}
         \omega'_1 = \omega_0 + \omega_2 + \cdots + \omega_n \\
         \omega'_2 = \omega_0 + \omega_1 + \omega_3 + \cdots + \omega_n \\
         \dots \\
         \omega'_n = \omega_0 + \omega_1 + \cdots + \omega_{n-1} \\
    \end{array}
    \right .
\]
It follows thence
\[ G = \{ \diag(t_0, \ t_1, \dots, \ t_n, \ t_0 t_2 \dots t_n, \dots, \ t_0 t_1 \dots t_{n-1}) \ | \ t_0, t_1, \dots, t_n \in \CC^{\times} \}. \]

Let us denote by \( T' \) and \( T \) the tori \( (\CC^{\times})^{\Sgmone} \) and \( (\CC^{\times})^n \) acting on the toric varieties \( Y \) and \( X \), respectively. Consider an action of a group \( \unipgr^n = \{ (c_1, \dots, c_n) \in \CC^n \} \) on \( Y \) given by formulas
\[
    \left .
    \begin{array}{cr}
    	x_i \mapsto x_i, & \ 0 \leq i \leq n, \\
    	\\ \medskip
    	x'_j \mapsto x'_j + c_j x_0 x_1 \dots \widehat{x_j} \dots x_n, & \ 1 \leq j \leq n. \\
    \end{array}
    \right .
    \label{action} \tag{\( \ast \ast \)}
\]
One can easily see that this action commutes with the action of the group \( G \). Therefore, we have an induced action on variety \( X \). In addition, \( \unipgr^n \)-orbits in \( X \) correspond to \( {G \times \unipgr^n \text{-orbits}} \) in \( Y \). Aside from that, as we can see from formulas (\ref{action}), the action of \( \unipgr^n \) on \( Y \) is normalized by the torus \( T' \). Hence, \( T \) permutes \( \unipgr^n \)-orbits in \( X \). Consider a point \( {x \in \CC^{\Sgmone}} \) with coordinates \( x_0 = x_1 = \dots = x_n = 1 \), \( x'_1 = \dots = x'_n = 0 \). The collection \( {C = \{ \rho'_1, \dots, \rho'_n \}} \) is not primitive, thus \( x \in Y \). Let \( \OO' \) be the \( G \times \unipgr^n \)-orbit of point \( x \). We have \( {\OO' = (\CC^{\times})^{n+1} \times \CC^n} \) and \( \dim(\OO') = 2n + 1 \), therefore the \( \unipgr^n \)-orbit \( \pi (\OO') = \OO' / G \) on \( X \) is of dimension \( n \). Thus, \( \OO = \pi (\OO') \) is an open dense orbit in \( X \). The second assertion of Theorem~\ref{main theorem} follows.
\end{proof}

Now we describe orbits of the additive action \( \GG_a^n \times X_{\Sigma} \to X_{\Sigma} \) that we constructed in the proof of Theorem~\ref{main theorem} in the case of \( n = 3 \).

\begin{proposition}
In terms of the aforementioned notation
\begin{enumerate}
\item the closed subvariety \( X_{\Sigma} \setminus \mathscr{O} \) consists of four irreducible components \( X_0, X_1, X_2, X_3 \);
\item the component \( X_0 \) is isomorphic to the projective plane \( \PP^2 \) with the trivial \( \addgr \)-action;
\item each component \( X_j, \ j = 1,2,3 \) is isomorphic to the blowup of the Hirzebruch surface \( \mathbb{F}_1 \) at a point; typical \( \addgr \)-orbits in \( X_j \) form a one-parameter family of one-dimensional orbits, the complement of this family is the union of curves of fixed points \( S_{jk}, \ {k \in \{ 0, 1, 2, 3 \} \setminus \{ j \} = \{ 0, k', k'' \}} \), where \( S_{jk} \simeq \PP^1 \) for all \( k \), and
\[ |S_{j0} \cap S_{jk'}| = |S_{j0} \cap S_{jk''}| = 1, \ S_{j0} \cap S_{jk'} \neq S_{j0} \cap S_{jk''}, \ S_{jk'} \cap S_{jk''} = \emptyset; \]
\item if \( j, l \in \{ 1, 2, 3 \}, \ j \neq l \), then the intersection of components \( X_j \) and \( X_l \) equals \( S_{jl} = S_{lj} \), the intersection of three components \( X_1, \ X_2 \) and \( X_3 \) is empty; moreover, the intersection of the component \( X_0 \) and a component \( X_j \) is a curve \( S_{j0} \), and the intersection of \( X_0 \) and a pair of components \( X_j \) and \( X_l \) is a single point, which is different for different pairs.
\end{enumerate}
\end{proposition}

\begin{proof}
We keep the notation that was used in the proof of Theorem~\ref{main theorem}. The complement \( {Y \setminus \OO'} \) is covered by irreducible subvarieties of dimension six
\[ Y_i = \{ x_i = 0 \} \setminus Z, \ 0 \leq i \leq 3; \]
thus, denoting \( X_i = \pi (Y_i) \), we obtain a decomposition into irreducible two-dimensional subvarieties \( X \setminus \OO = \displaystyle\bigcup_{i=0}^3 X_i \). Assertion (1) follows.

To acquire a fan of a variety \( X_i \), we take the quotient \( N / \ZZ b_i \) and consider the projections of all cones that contain the vector \( b_i \). Since \( N = \ZZ b_0 \oplus \ZZ a_1 \oplus \ZZ a_2 \), we have \( N / \ZZ b_0 = \ZZ a_1 \oplus \ZZ a_2 \). The canonical projection is given by
\[ a_1 \mapsto (1, 0), \ a_2 \mapsto (0, 1), \ a_3 \mapsto (-1, -1), \]
\[ b_0 \mapsto (0, 0), \ b_1 \mapsto (1, 0), \ b_2 \mapsto (0, 1), \ b_3 \mapsto (-1, -1). \]
We get a fan in \( \ZZ^2 \), which is formed by vectors \( (1, 0), \ (0, 1), \ (-1, -1) \). Thus, \( X_0 \simeq \PP^2 \). Since the fan \( \Sigma \) is invariant under coordinate permutations on \( N \), we get \( X_1 \simeq X_2 \simeq X_3 \). Notice that \( N = \ZZ b_3 \oplus \ZZ a_2 \oplus \ZZ a_3 \). Therefore, \( N / \ZZ b_3 = \ZZ a_2 \oplus \ZZ a_3 \) and the projection is
\[ a_1 \mapsto (-1, 0), \ a_2 \mapsto (1, 0), \ a_3 \mapsto (0, 1), \]
\[ b_0 \mapsto (0, -1), \ b_1 \mapsto (-1, -1), \ b_2 \mapsto (1, -1), \ b_3 \mapsto (0, 0). \]
The acquired two-dimensional fan is formed by vectors \( (0, 1), (-1, 0), (-1, -1), (0, -1), (1, -1) \). We reflect the lattice \( \ZZ^2 \) through the vertical axis and then through the horizontal axis. The transformed fan corresponds to the blowup of the Hirzebruch surface at a fixed point that corresponds to the first quadrant of the plane.

\smallskip

One can easily see that
\[ \mathscr{C}(\Sigma) = \{ \{ \rho'_1, \rho_0 \}, \{ \rho'_2, \rho_0 \}, \{ \rho'_3, \rho_0 \}, \{ \rho'_1, \rho_2 \}, \{ \rho'_2, \rho_3 \}, \{ \rho'_3, \rho_1 \}, \{ \rho_1, \rho_2, \rho_3 \} \}. \label{primitive collections} \tag{\( \ast \ast \ast \)} \]
By formulas~(\ref{action}) we get that if \( x_0 = 0 \) then the group \( \addgr \) acts trivially, hence \( X_0 \) consists of fixed points. Assertion (2) follows.

If \( x \in Y_1 \) then \( x'_3 \neq 0 \) by equality ~(\ref{primitive collections}). Consider an open \( G \)-invariant subset \( U'_1 \) of the variety \( Y_1 \), satisfying conditions \( x_0 \neq 0, \ x_2 \neq 0, \ x_3 \neq 0 \), that is \( U'_1 = Y_1 \cap \{ x_0 x_2 x_3 \neq 0 \} \). The kernel of the \( \addgr \)-action on \( U'_1 \) is two-dimensional, because it changes the coordinate \( x'_1 \) only. The decomposition into irreducible components of the complement of \( U'_1 \) in \( Y_1 \) is \( S'_{10} \cup S'_{12} \cup S'_{13} \), where \( S'_{1k} = Y_1 \cap \{ x_k = 0 \} \). Notice that
\[ S'_{10} \cap S'_{12} = G \cdot (0, 0, 0, 1, 1, 1, 1), \ S'_{10} \cap S'_{13} = G \cdot (0, 0, 1, 0, 1, 1, 1), \ S'_{12} \cap S'_{13} = \emptyset. \]
Let \( S_{1k} = \pi(S'_{1k}) \). The variety \( S'_{1k} \) is smooth, hence normal. Since the quotient preserves normality, and an algebraic curve is normal if and only if it is smooth, we have \( {\PP^1 \simeq S_{10} \simeq S_{12} \simeq S_{13}} \).

Similarly, we get \( Y_2 = U'_2 \sqcup (S'_{20} \cup S'_{21} \cup S'_{23}) \) and \( Y_3 = U'_3 \sqcup (S'_{30} \cup S'_{31} \cup S'_{32}) \), where \( \displaystyle U'_j = Y_j \cap \Big\{ \prod_{\substack{k=0 \\ k \neq j}}^3 x_k \neq 0 \Big\}, \ S'_{jk} = Y_j \cap \{ x_k = 0 \}. \) Therefore,
\[ Y_1 \cap Y_2 = S'_{12}, \ Y_1 \cap Y_3 = S'_{13}, \ Y_2 \cap Y_3 = S'_{23}, \ Y_1 \cap Y_2 \cap Y_3 = \emptyset; \]
\[ {Y_0 \cap Y_j = S'_{j0}}, \ {|\pi(Y_0 \cap Y_j \cap Y_l)| = 1}. \]
Assertions (3)-(4) follow if we let \( S_{jk} = \pi(S'_{jk}) \).
\end{proof}


{}

\end{document}